\documentclass[11pt]{amsart}

\usepackage{amssymb,amsmath,xcolor,hyperref}
\usepackage[normalem]{ulem}
\usepackage[mathscr]{eucal}
\usepackage[all]{xy}

\theoremstyle{plain}
\newtheorem{thm}{Theorem}[section]

\newtheorem{lem}[thm]{Lemma}

\newtheorem{cor}[thm]{Corollary}
\newtheorem{prop}[thm]{Proposition}

\theoremstyle{definition}

\newtheorem{ex}[thm]{Example}

\theoremstyle{remark}

\newtheorem*{remark*}{Remark}

\newtheorem{ass}{Assumption}

\numberwithin{equation}{section}

\newcommand{\cU}{{\mathcal{U}}}

\newcommand{\cW}{{\mathcal{W}}}

        \newcommand{\field}[1]{{\mathbb{#1}}}

        \newcommand{\RR}{\field{R}}
        \newcommand{\CC}{\field{C}}

\newcommand{\supp}{\operatorname{supp}}

\newcommand{\Dom}{\mbox{\rm Dom}}

\newcommand{\id}{\mbox{\rm id}}

\allowdisplaybreaks

\begin{document}

\title[A vanishing theorem in $K$-theory for spectral projections]{A vanishing theorem in $K$-theory for spectral projections of a non-periodic magnetic Schr\"odinger operator}

\author[Y. A. Kordyukov]{Yuri A. Kordyukov}

\address{Institute of Mathematics, Ufa Federal Research Centre, Russian Academy of Sciences, 112~Chernyshevsky str., 450008 Ufa, Russia 
}  

\email{yurikor@matem.anrb.ru}

\author[V. M. Manuilov]{Vladimir M. Manuilov}

\address{Moscow Center for Fundamental and Applied Mathematics, Moscow State University,
Leninskie Gory 1, 119991 Moscow, Russia}

\email{manuilov@mech.math.msu.su}




\begin{abstract}
 We consider the Schr\"odinger operator $H(\mu) = \nabla_{\bf A}^*\nabla_{\bf A} + \mu  V$ on a Riemannian manifold $M$ of bounded geometry, where $\mu>0$ is a coupling parameter, the magnetic field ${\bf B}=d{\bf A}$ and the electric potential $V$ are uniformly $C^\infty$-bounded, $V\geq 0$. We assume that, for some $E_0>0$, each connected component of the sublevel set $\{V<E_0\}$ of the potential $V$ is relatively compact. Under some assumptions on geometric and spectral properties of the connected components, we show that, for sufficiently large $\mu$, the spectrum of $H(\mu)$ in the interval $[0,E_0\mu]$ has a gap, the spectral projection of $H(\mu)$, corresponding to the interval $(-\infty,\lambda]$ with $\lambda$ in the gap, belongs to the Roe $C^*$-algebra $C^*(M)$ of the manifold $M$, and,  if $M$ is not compact, its class in the $K$ theory of $C^*(M)$ is trivial. 
\end{abstract}

\date{}

 \maketitle
\section{Introduction}
Let $(M,g)$ be a Riemannian manifold of bounded geometry. This means that the curvature $R^{TM}$ of the Levi-Civita connection $\nabla^{TM}$ and its covariant derivatives of any order are uniformly bounded on $M$ in the norm induced by $g$, and the injectivity radius $r_M$ of $(M, g)$ is positive. In particular, $M$ is complete.

Let $\bf B$ be a closed differential 2-form on $M$. We assume that ${\bf B}\in C^\infty_b\Omega^2(M)$, which means that $\bf B$ and its covariant derivatives of any order are uniformly bounded on $M$ in the norm induced by $g$. 
We assume that $\bf B$ is exact. Choose a 1-form $\bf A$ on $M$ such that $d{\bf
A} = \bf B$. As in geometric quantization we may regard $\bf A$ as
defining a Hermitian connection $\nabla_{\bf A} = d+i{\bf A}$ on
the trivial complex line bundle $\mathcal L$ over $M$, whose curvature is
$i\bf B$. Physically we can think of $\bf A$ as the magnetic vector potential for the magnetic field $\bf B$.

Consider the magnetic Schr\"odinger operator
given by
\begin{equation}\label{ham}
H(\mu) = \nabla_{\bf A}^*\nabla_{\bf A} + \mu  V,
\end{equation}
where $V$ is a real-valued smooth function on $M$ and $\mu>0$ is the coupling constant.
We assume that $V\in C^\infty_b(M)$, which means that $V$ and its covariant derivatives of any order are uniformly bounded on $M$ in the norm induced by $g$. We also assume that $V(x) \geq 0$ for any $x\in M$.  

\begin{ass}\label{a1}
There exists $E_0>0$ such that each connected component of the sublevel set 
\[
U = \{x\in M\,:\,  V(x) < E_0\}
\]
is a relatively compact domain (with smooth boundary) in $M$. { Moreover, if we denote by $D$  the set of connected components of $U$ and, for any $h \in D$, by $U_{h}$ the corresponding connected component, then 
\begin{equation}\label{e:tr0}
\sup_{h\in D}{\rm diam}\,(U_h)<\infty,
\end{equation}
\begin{equation}\label{e:tr1}
\inf_{h_1,h_2\in D, h_1\neq h_2}{\rm dist}\,(U_{h_1}, U_{h_2})>0,
\end{equation} 
and
\begin{equation}\label{e:E0m}
E_0>\sup_{h\in D}m_{U_{h}}(V),
\end{equation}
where
\[
m_{U_h}(V)=\min \{V(x)\, :\, x\in U_h\}.
\]
%
}
\end{ass}

It follows from bounded geometry of $M$ that for any $R>0$ the number of $U_h$ in each ball of radius $R$ is uniformly bounded.  

{
From now on, we fix $E_0>0$, satisfying Assumption~\ref{a1}. 
}

Consider the Dirichlet realization $H_{U}(\mu)$ of the operator $H(\mu)$ in $U$. { Since
\[
U=\bigsqcup_{h \in D}U_{h},
\] 
} the operator $H_U(\mu)$ has the form
\[
H_U(\mu)=\bigoplus_{h\in D} H_{U_h}(\mu)
\]
with respect to the direct sum decomposition $L^2(U)=\oplus_{h \in D}L^2(U_{h})$,
where $H_{U_h}(\mu)$ is the Dirichlet realization of the operator $H(\mu)$ in $U_h$.  Each operator $H_{U_h}(\mu)$, $h\in D$, has discrete spectrum. If we denote by $\sigma(H_U(\mu))$ the spectrum of the operator $H_U(\mu)$ in $L^2(U)$ and by $\sigma(H_{U_h}(\mu))$ the spectrum of the operator $H_{U_h}(\mu)$ in $L^2(U_h)$, then 
\[
\sigma(H_U(\mu))=\overline{\bigcup_{h\in D}\sigma(H_{U_h}(\mu))}.
\]

\begin{ass}\label{a3}
{For the fixed $E_0>0$, satisfying Assumption~\ref{a1}, there exist $E_1\in (0,E_0)$ and $\mu_0>0$ such that, for any $\mu>\mu_0$, the spectrum of the operator $H_U(\mu)$ in the interval $[0,E_1\mu]$ has a gap $(a(\mu),b(\mu))$ such that}
\begin{equation}
\label{e:size}
\liminf_{\mu\to+\infty} \mu^{-1/2}\log (b(\mu)-a(\mu))\geq 0.
\end{equation} 
\end{ass}

Equivalently, the latter condition means that for any $\epsilon>0$, there exist $\mu_\epsilon>0$ and $c_\epsilon>0$ such that, for any $\mu>\mu_\epsilon$, we have 
\[
b(\mu)-a(\mu)>c_\epsilon e^{-\epsilon \mu^{1/2}}.
\]  

{By a gap in the spectrum $\sigma(A)$ of a self-adjoint operator $A$, we understand a bounded connected component of $\mathbb R\setminus \sigma(A)$.}

Recall the definition of the Roe algebra of $M$ \cite{Roe}. 
For a Hilbert space $H$ we write $\mathbb B(H)$ (resp., $\mathbb K(H)$) for the algebra of all bounded (resp., all compact) operators on $H$. 
Consider the standard action of the algebra $C_0(M)$ of continuous functions on $M$ vanishing at infinity on $L^2(M)$ by multiplication. An operator $T\in\mathbb B(L^2(M))$ is {\it locally compact} if the operators $Tf$ and $fT$ are compact for any $f\in C_0(M)$. It has {\it finite propagation} if there exists some $R>0$ such that $fTg=0$ whenever the distance between the supports of $f,g\in C_0(M)$ is greater than $R$. The {\it Roe algebra} $C^*(M)$ is the norm completion of the $*$-algebra of locally compact, finite propagation operators on $L^2(M)$. 

The operator $H(\mu)$ is essentially self-adjoint in $L^2(M)$ with initial domain $C^\infty_c(M)$. Denote by $\sigma(H(\mu))$ the spectrum of the operator $H(\mu)$. For $\lambda\in \RR$, let $E_{H(\mu)}(\lambda) = {\chi}_{(-\infty, \lambda]}(H(\mu))$ denote the spectral projection of $H(\mu)$, corresponding to the interval $(-\infty,\lambda]$.  

The main result of the paper is the following theorem.

\begin{thm} \label{t:trivial}
Let Assumptions \ref{a1} and \ref{a3} hold and let $(a(\mu), b(\mu))$, $\mu>\mu_0$, be the gap in the spectrum of the operator $H_U(\mu)$ in the interval $[0,E_1\mu]$ given by Assumption \ref{a3}. Then there exist  $C>0$, $c>0$, and $\mu_1>\mu_0$ such that for any $\mu>\mu_1$ the interval $$(a(\mu)+Ce^{-c\mu^{1/2}},b(\mu)-Ce^{-c\mu^{1/2}})$$
 is not in the spectrum of $H(\mu)$. Moreover, for any $\lambda$ in this interval, the spectral projection $E_{H(\mu)}(\lambda)$ belongs to $C^*(M)$ and, if $M$ is not compact, satisfies
\[
[E_{H(\mu)}(\lambda)] = 0 \in K_0(C^*(M)).
\]
 \end{thm}

\begin{ex}
Assumption~\ref{a3} is a complicated spectral condition on the operator $H_U(\mu)$, and it seems impossible to check it for a general potential $V$ satisfying Assumption~\ref{a1}. It is also quite difficult to provide concrete examples of operators $H(\mu)$, satisfying Assumption~\ref{a3}. The only class of such operators, which we know, is described as follows.

Fix $h_0\in D$. Assume that, for any $h\in D$, the domain $U_h$ is simply connected and there exists an isometry $f_h : U_{h_0} \to U_h$, which preserves the magnetic field $\mathbf B$ and the potential $V$:  
\[
f^*_h (\mathbf B\left|_{U_h}\right.)=\mathbf B\left|_{U_{h_0}}\right.,\quad f^*_h (V\left|_{U_h}\right.)=V\left|_{U_{h_0}}\right.. 
\]
By gauge invariance, the operators $H_{{U_h}}(\mu)$, $h\in D$, are unitarily equivalent to the operator $H_{{U_{h_0}}}(\mu)$, and, therefore, have the same spectrum. It follows that $\sigma(H_U(\mu))=\sigma(H_{U_{h_0}}(\mu))$ is a discrete set. Given $E_1$ and $E_2$ such that $0<E_1<E_2<E_0$, by a simple quasimode construction, one can show that $\sigma(H_{U_{h_0}}(\mu))\cap [E_1\mu,E_2\mu]\neq \emptyset$ and $\sigma(H(\mu))\cap [E_1\mu,E_2\mu]\neq \emptyset$ for sufficiently large $\mu$. Using the facts that the spectrum of $H(\mu)$ on the interval $[E_1\mu,E_2\mu]$ is localized inside an exponentially small neighborhood of the spectrum of $H_U(\mu)$ and the cardinality of $\sigma(H_{U_{h_0}}(\mu))\cap [E_1\mu,E_2\mu]$ grows polynomially in $\mu$ (see Theorem~\ref{t:D} and Lemma~\ref{l:Weyl} below), one can easily prove the existence of gaps $(a(\mu),b(\mu))$ in the spectrum of $H(\mu)$ on the interval $[E_1\mu,E_2\mu]$, satisfying \eqref{e:size}, and, in particular, verify Assumption~\ref{a3}. Moreover, by Theorem~\ref{t:trivial}, we infer that, for any $E_1\in (0,E_0)$, there exists $\lambda(\mu)\in [E_1\mu, E_0\mu)$ for sufficiently large $\mu$ such that the projection $E_{H(\mu)}(\lambda(\mu))$ is non-zero, belongs to $C^*(M)$ and,  if $M$ is not compact, satisfies
\[
[E_{H(\mu)}(\lambda(\mu))] = 0 \in K_0(C^*(M)).
\]
\end{ex} 

Theorem~\ref{t:trivial} is partly motivated by noncommutative geometry approach to the study of topological insulators, in particular, to the integer quantum Hall effect, initiated by Bellissard \cite{Bel86,Bel94}. In physics, $\lambda$ is typically called the Fermi energy of the physical system described by the quantum Hamiltonian $H(\mu)$ and the corresponding spectral projection $E_{H(\mu)}(\lambda)$ is referred to as the Fermi projection. The fact that the energy value $\lambda$ is in a spectral gap of $H(\mu)$ means that at this energy we have an insulator, and it is a topological insulator if  the class of the corresponding Fermi projection in $K$-theory is non-trivial. So, by Theorem~\ref{t:trivial}, in our setting the insulator is not a topological one. 

Such a vanishing result first was proved by Nakamura and Bellissard \cite{Nak+Bel} for the Euclidean plane equipped with uniform magnetic field and periodic potential. Actually, the main purpose in \cite{Nak+Bel} was to show vanishing of the quantum Hall conductance. This result provides a rigorous mathematical treatment of a physical observation made by Tesanovic, Axel and Halperin \cite{TAH} that low energy bands in an ordered or slightly disordered 2D crystal submitted to a uniform magnetic field, do not contribute to the quantum Hall conductance. Recall that, by the Kubo formula, the Hall conductance in the usual model
of the integer quantum Hall effect on the Euclidean plane, also called the Chern number of the Fermi projection, can be naturally interpreted as a pairing of a cyclic 2-cocycle defined on some dense subalgebra of the $C^*$-algebra of observables with the $K$-theory class of the Fermi projection. A similar fact holds for the model of the fractional quantum Hall effect on the hyperbolic plane suggested in \cite{CHMM, MM}. On the other hand, such a cyclic formula can be only derived for integer invariants. It is impossible to have it for torsion invariants, for instance, for the $\mathbb Z_2$-invariant associated to the time-reversal invariant systems (see, for instance, the Kane-Mele model \cite{KM05}). Therefore, a cyclic formula does not arise in general topological insulator systems and instead one should deal with the $K$-theory class of the Fermi projection and $K$-theoretic index pairing directly.

The results of \cite{Nak+Bel} were extended by the first author, Mathai, and Shubin \cite{kms} to the case of an arbitrary regular covering of a compact manifold, periodic magnetic field and electric potential. 

Our motivation is to extend the results of \cite{Nak+Bel,kms} to the non-periodic setting of magnetic Schr\"odinger operators with disordered potentials. As observed by Bellissard \cite{Bel86,Bel94}, non-commutative $C^*$-algebras of observables are needed in order to study disordered systems. It has been realized fairly recently (see, for instance, \cite{EM19,Ku17,KLT22,L23,LT21,LT22}) that Roe algebras, which come from the mathematical subject of coarse geometry, are a particularly good, physically well-motivated
choice here. This explains our choice of the $C^*$-algebra of observables and use of some notions and results from coarse geometry.

\section{Outline of the proof}\label{s:outline}

In this section, we outline  the proof of Theorem~\ref{t:trivial}. As in \cite{Nak+Bel,kms}, it is based on semiclassical approximation and noncommutative geometry tools, but concrete approaches are different. First, we show that the spectrum of the operator $H(\mu)$ on the interval $[0,E_1\mu]$, with $E_1<E_0$, is exponentially close to the spectrum of the Dirichlet realization $H_U(\mu)$ of the operator $H(\mu)$ in the sublevel set $U=\{x\in M\,:\,  V(x) < E_0\}$. 

\begin{thm}\label{t:D}
Under Assumption \ref{a1}, for any $E_1\in (0,E_0)$, there exist $C>0$, $c>0$, and $\mu_2>0$ such that for
any $\mu>\mu_2$
\[
\sigma(H(\mu))\cap [0, E_1 \mu ] \subset \{\lambda\in [0,
E_1\mu ] : {\rm dist}(\lambda, \sigma(H_U(\mu)))<
Ce^{-c\mu^{1/2} }\}.
\]
\end{thm}

The proof of Theorem~\ref{t:D} is given in Section~\ref{s:D}.

Our approach to semiclassical approximation is different from the approaches of \cite{Nak+Bel,kms}. We follow the approach to the study of the tunneling effect in multi-well problems developed by Helffer and Sj\"ostrand for Schr\"odinger operators with electric potentials (see for instance \cite{HSI,HSII}) and extended to magnetic Schr\"odinger operators in \cite{HS87,HM96}. Since $H(\mu)$ is not with compact resolvent, we work not with individual eigenfunctions as in \cite{HSI}, but with resolvents, using the strategy developed in \cite{HSII,HS88,Di-Sj,Carlsson} for the case of electric potential and in \cite{Frank,HK08,higherLL} for the case of magnetic field.

Next, we construct semiclassical approximation at the level of spectral projections. We suppose that  Assumption \ref{a1} holds with the fixed $E_0$ and $E_1$ satisfies
\begin{equation}\label{e:epsilon1}
\sup_{h\in D}m_{U_h}(V)<E_1<E_0.
\end{equation} 
For any $E<E_0$, we set 
\begin{equation} \label{e:UEh}
U_{E,h}= \{x\in U_h\,:\,  V(x) < E\}.
\end{equation}
The condition~\eqref{e:epsilon1} ensures that $U_{E_1,h}\neq \emptyset$ for any $h\in D$.

Take any $\eta>0$ such that $E_1+3\eta<E_0$ and, for any $h\in D$, fix a function $\phi_h\in C_c^\infty(U_h)$ such that $\supp \phi_h \subset U_{E_1+2\eta,h}$, $\phi_h\equiv 1$ on $U_{E_1+\eta,h}$, and the family $\{\phi_h : h\in D\}$ is bounded in $C^\infty_b(M)$ (we refer to the proof of Proposition \ref{p:D} for a construction of such a family).

For $\lambda\in \RR$, let $E_{H_{U_h}(\mu)}(\lambda) = {\chi}_{(-\infty, \lambda]}(H_{U_h}(\mu))$ denote the spectral projection of $H_{U_h}(\mu)$, corresponding to the interval $(-\infty,\lambda]$. The image of the operator $\phi_h E_{H_{U_h}(\mu)}(\lambda)$ is a finite dimensional subspace of $L^2(U_h)$, which can be considered as a finite dimensional subspace $\mathcal H_{U_h,\mu} (\lambda)$ of $L^2(M)$ for any $\mu>0$. Set 
\[
\mathcal H_{U,\mu}(\lambda)=\bigoplus_{h\in D} \mathcal H_{U_h,\mu} (\lambda)\subset L^2(M)
\]
and denote by $P_{\mathcal H_{U,\mu}(\lambda)}$ the orthogonal projection on $\mathcal H_{U,\mu}(\lambda)$ in $L^2(M)$:
\[
P_{\mathcal H_{U,\mu}(\lambda)}=\sum_{h\in D}P_{\mathcal H_{U_h,\mu} (\lambda)},
\]
where $P_{\mathcal H_{U_h,\mu} (\lambda)}$ is the orthogonal projection on $\mathcal H_{U_h,\mu} (\lambda)$ in $L^2(U_h)$.

\begin{thm} \label{t:equiv}
Let Assumption \ref{a1} hold with the fixed $E_0$ and let $E_1$ satisfy \eqref{e:epsilon1}.  Let $\lambda(\mu) \in (0,E_1\mu)$ be such that, for any $\epsilon >0$\,, there exists $C_\epsilon>0$ such that 
\[
{\rm dist}\,(\lambda(\mu), \sigma(H_{U}(\mu)))\geq
\frac{1}{C_\epsilon}e^{-\epsilon \mu^{1/2}}
\]
for all sufficiently large $\mu$. Then there exists  a constant $\mu_3 >0$ such that for all $\mu> \mu_3$, the projections $E_{H(\mu)}(\lambda)$ and $P_{\mathcal H_{U,\mu} (\lambda)}$ are in $C^*(M)$ and are Murray-von Neumann equivalent in $C^*(M)$. In particular,
\[
[E_{H(\mu)}(\lambda)]=[P_{\mathcal H_{U,\mu} (\lambda)}] \in
K_0(C^*(M)).
\]
 \end{thm}

The proof of Theorem~\ref{t:equiv} is given in Section~\ref{s:equiv}.

Finally, we suppose that Assumptions \ref{a1} and \ref{a3} hold and $E_1$ and $(a(\mu),b(\mu))$ are given by Assumption \ref{a3}. {Without loss of generality, we may assume that $E_1$ satisfies \eqref{e:epsilon1}}. With the constants $C, c, \mu_2>0$ given by Theorem \ref{t:D}, we infer that for any $\mu>\max(\mu_0,\mu_2)$ the interval $(a(\mu)+Ce^{-c\mu^{1/2}},b(\mu)-Ce^{-c\mu^{1/2 }})$ is not in the spectrum of $H(\mu)$. Moreover, by Theorem~\ref{t:equiv}, for any $\lambda$ in this interval and $\mu$ sufficiently large, the corresponding spectral projection $E_{H(\mu)}(\lambda)$ belongs to $C^*(M)$ and is Murray-von Neumann equivalent to the projection $P_{\mathcal H_{U,\mu} (\lambda)}$. In particular,
\[
[E_{H(\mu)}(\lambda)]=[P_{\mathcal H_{U,\mu} (\lambda)}] \in K_0(C^*(M)).
\]

To complete the proof of Theorem~\ref{t:trivial},  we show that, if $M$ is not compact, the  class $[P_{\mathcal H_{U,\mu} (\lambda)}] \in K_0(C^*(M))$ is trivial. Unlike the periodic setting, the triviality of this class is not immediate and quite delicate. Here we slightly extend our previous results of \cite{wannier} on triviality of generalized Wannier projections associated with discrete subsets of manifolds of bounded geometry. 

We consider the following, more general setting.  
Let $\mathcal U=\{\mathcal U_{h} : h \in D\}$ be a family of relatively compact domains (with smooth boundary) in $M$.  Assume that  the diameters of $\mathcal U_{h}, h \in D$, are uniformly bounded:
\begin{equation}\label{e:udiam}
\sup_{h\in D}{\rm diam}\,(\mathcal U_{h})=\delta<\infty,
\end{equation}
and the family is uniformly discrete:
\begin{equation}\label{e:udiscr}
\inf_{h_1,h_2\in D, h_1\neq h_2}{\rm dist}\,(\mathcal U_{h_1}, \mathcal U_{h_2})=\epsilon>0.
\end{equation}
Note that here we don't assume that the family $\mathcal U$ is related with the potential $V$.

For any $h\in D$, let $H_h$ be a finite-dimensional subspace in $L^2(M)$, $\dim H_h=n_h$, such that each $\phi\in H_h$ is supported in $\mathcal U_h$. We assume that 
\begin{equation}
\label{e:defn}
n:=\max_{h\in D}n_h <\infty.
\end{equation}

Let $H_{\mathcal U}\subset L^2(M)$ be the closure of the linear span of the union of all $H_h$, $h\in D$. Thus, we have a direct sum decomposition
\[
H_{\mathcal U}=\bigoplus_{h\in D} H_h.
\]
Let $P_{H_{\mathcal U}}$ denote the  orthogonal projection onto $H_{\mathcal U}$.  

\begin{thm}\label{t:PHU}
The projection $P_{\mathcal H_{\cU}}$ is in $C^*(M)$ and, if $M$ is not compact, 
\[
[P_{\mathcal H_{\cU}}] = 0
\in K_0(C^*(M)).
\]
\end{thm}

The proof of Theorem~\ref{t:PHU} is given in Section~\ref{s:PHU}. 

Application of Theorem~\ref{t:PHU} in the case $\mathcal U_h=U_h$ and $H_h = \mathcal H_{U_h,\mu} (\lambda)$ completes the proof of Theorem~\ref{t:trivial}. The condition~\eqref{e:defn} holds in this case due to Lemma \ref{l:Weyl}.

{Note that if $M$ is compact then $C^*(M)$ is the $C^*$-algebra $\mathbb K(L^2(M))$ of compact operators on $L^2(M)$, $K_0(C^*(M))\cong \mathbb Z$, and the class $[P_{\mathcal H_{\cU}}]$ is just the rank of $P_{\mathcal H_{\cU}}$. In this case, Theorem~\ref{t:PHU} does not hold: $[P_{\mathcal H_{\cU}}]$ does not vanish but it encodes no non-trivial geometric information.}

\section{Proof of Theorem~\ref{t:D}}\label{s:D}
\subsection{Weighted $L^2$ spaces} 
Let $W$ be an open domain (with smooth boundary) in $M$.
Denote by $C^{0,1}(\overline{W},\RR)$ the class of uniformly
Lipschitz continuous, real-valued functions on $\overline W$. 
For any $\Phi\in C^{0,1}(\overline{W},\RR)$ and $\mu>0$ define the Hilbert
space
\[
L^2_{\mu^{1/2}\Phi}(W)=\{u\in L^2_{loc}(W) : e^{\mu^{1/2}\Phi}u\in
L^2(W)\}
\]
with the norm
\[
\|u\|_{\mu^{1/2}\Phi}=\|e^{\mu^{1/2}\Phi} u\|, \quad u\in
L^2_{\mu^{1/2}\Phi}(W),
\]
where $\|\cdot\|$ denotes the norm in $L^2(W)$:
\[
\|u\|=\left(\int_W|u(x)|^2\,dx\right)^{1/2}, \quad u\in L^2(W).
\]
By $\|\cdot\|_{\mu^{1/2}\Phi}$ we will also denote the norm of a
bounded operator in $L^2_{\mu^{1/2}\Phi}(W)$.

Denote by $H_W(\mu)$ the Dirichlet realization of the operator $H(\mu)$ in $W$.
Recall the following important identity (cf. for instance \cite[Theorem 1.1]{HS87}).

\begin{lem}\label{l:1.1}
Let $W\subset M$ be an open domain (with $C^2$ boundary) and
$\Phi\in C^{0,1}(\overline{W},\RR)$. For any $\mu>0$, $z\in\CC$ and
$u\in \Dom (H_W(\mu))$ one has
\begin{multline}\label{e:energy}
{\rm Re}\, \int_{W} e^{2\mu^{1/2}\Phi} (H(\mu)-z)u \cdot \bar u\,dx  =
 \int_W |\nabla_{\bf A}(e^{\mu^{1/2}\Phi}u)|^2\,dx \\ +\int_{W} e^{2\mu^{1/2}\Phi}(\mu(V-
|\nabla\Phi|^2) - {\rm Re}\, z)|u|^2\,dx.
\end{multline}
\end{lem}

\subsection{Estimates away from the wells}
 Let
\[
m_W(V)=\inf \{V(x)\, :\, x\in W\}.
\]  
For any $E\geq m_W(V)$, consider the (degenerate) Agmon metric
\[
g_E=[V(x)-E]_+\cdot g,
\]
where, for any $\lambda\in \RR$, $\lambda_+=\max(\lambda,0)$. Let $d_E(x,y)$ be the associated
distance function on $W$.

Introduce the following class of weights
\[
\cW_E(\overline{W})=\{\Phi\in C^{0,1}(\overline{W},\RR) :
\underset{x\in \overline{W}} {\operatorname{ess-inf}} (V(x)-E-|\nabla\Phi(x)|^2)
>0\}.
\]
Examples of functions in the class $\cW_E(\overline{W})$ are given
by the functions $\Phi(x)=(1-\epsilon)d_E(x,X)$, with 
$0<\epsilon \leq 1$ and $X\subset W$ and $\Phi(x)=d_{E_1}(x,X)$, with 
$E_1>E$ and $X\subset W$. 

 Let $W\subset  M$ be an open domain (with a smooth boundary) such that $m_{W}(V)>0$.

\begin{prop}\label{p:1}
Let $\Phi\in \cW_E(\overline{W})$. Assume that $K(\mu)$ is a bounded
subset in $\CC$ such that $K(\mu)\subset \{z\in \CC : {\rm Re}\,z < \mu 
(E-\alpha)\}$ for some $E\geq m_W(V)$ and $\alpha>0$. If $\mu>0$ is large enough,
then $K(\mu)\cap \sigma(H_W(\mu))=\emptyset\,$, and for any $z\in K(\mu)$
the operator $(H_W(\mu)-z)^{-1}$ defines a bounded operator in
$L^2_{\mu^{1/2}\Phi}(W)$ with
\[
\|(H_W(\mu)-z)^{-1}\|_{\mu^{1/2}\Phi}\leq \frac{1}{\alpha \mu}, \quad z\in K(\mu).
\]
\end{prop}

\begin{proof}
By  Lemma~\ref{l:1.1}, for any $z\in K(\mu)$ and
$u\in \Dom (H_W(\mu))$,
we have
\begin{multline*}
{\rm Re}\, \int_{W} e^{2\mu^{1/2}\Phi} (H(\mu)-z)u \cdot \bar u\,dx  
\\ 
\begin{aligned}
& \geq \int_{W} e^{2\mu^{1/2}\Phi}(\mu (V-
|\nabla\Phi|^2)- {\rm Re}\, z)|u|^2\,dx \\ & \geq \int_{W} e^{2\mu^{1/2}\Phi}(\mu E- {\rm Re}\, z)|u|^2\,dx 
\geq \alpha \mu\|u\|^2_{\mu^{1/2}\Phi},
\end{aligned}
\end{multline*}
that immediately completes the proof.
\end{proof}

\begin{cor}\label{c:h1}
Under the assumptions of Proposition~\ref{p:1}, we have
\[
\|\nabla_{\bf A}(e^{\mu^{1/2}\Phi}(H_W(\mu)-z)^{-1}v)\|^2 \leq
\frac{C}{\mu}\|v\|^2_{\mu^{1/2}\Phi},\quad v\in
L^2_{\mu^{1/2}\Phi}(W).
\]
\end{cor}

\begin{proof}
By~(\ref{e:energy}), for any $\mu$ large enough, one has
\begin{multline*}
 \int_W |\nabla_{\bf A}(e^{\mu^{1/2}\Phi}(H_W(\mu)-z)^{-1}v)|^2\,dx={\rm Re}\, \int_{W} e^{2\mu^{1/2}\Phi} v\cdot \overline{(H_W(\mu)-z)^{-1}v}\,dx \\ -\int_{W} e^{2\mu^{1/2}\Phi}(\mu (V-|\nabla\Phi|^2)- {\rm Re}\, z)|(H_W(\mu)-z)^{-1}v|^2\,dx.
\end{multline*}

We know that ${\rm Re}\,z <  \mu (E-\alpha)$, $V$ and $|\nabla\Phi|$ are uniformly bounded.  For the last term, we have
\begin{multline*}
-\int_{W} e^{2\mu^{1/2}\Phi}(\mu(V-|\nabla\Phi|^2)- {\rm Re}\, z)|(H_W(\mu)-z)^{-1}v|^2\,dx \\ 
\begin{aligned} 
= & \int_{W} e^{2\mu^{1/2}\Phi}({\rm Re}\, z-\mu V)|(H_W(\mu)-z)^{-1}v|^2\,dx\\ & + \mu\int_{W} e^{2\mu^{1/2}\Phi} |\nabla\Phi|^2|(H_W(\mu)-z)^{-1}v|^2\,dx\\
& \leq C\mu\|(H_W(\mu)-z)^{-1}v \|^2_{\mu^{1/2}\Phi}\leq \frac{C_1}{\mu}\|v \|^2_{\mu^{1/2}\Phi}.
\end{aligned}
\end{multline*}

For the first term, we have
\begin{multline*}
{\rm Re}\,( e^{2\mu^{1/2}\Phi} v, (H_W(\mu)-z)^{-1}v) \\ \leq
\frac{1}{2}\left(\mu^{-1}\|v\|^2_{\mu^{1/2}\Phi} +
\mu\|(H_W(\mu)-z)^{-1}v\|_{\mu^{1/2}\Phi}^2\right) 
\leq
\frac{C_2}{\mu} \|v\|^2_{\mu^{1/2}\Phi},
\end{multline*}
that completes the proof.
\end{proof}

\subsection{Estimates near the wells}
In this section, we assume that Assumption \ref{a1} is satisfied with the fixed $E_0$ and $E_1$ satisfies \eqref{e:epsilon1}. Fix some $E_2$ and $E_3$ such that $E_1<E_2<E_3<E_0$ and consider a weight function $\Phi_h\in \cW_{E_1}(U_h)$ given by $\Phi_h(x)=d_{E_2}(x,U_{E_3,h})$, where $U_{E_3,h}$ is defined by \eqref{e:UEh}. 

\begin{prop}\label{p:10}
Assume that $K(\mu)$ is a bounded subset in $\CC$ such that
$K(\mu)\subset \{z\in \CC : {\rm Re}\,z <   
 E_1\mu \}$ and, for any $\epsilon >0$\,, there exists $C_\epsilon>0$ such that 
\[
{\rm dist}\,(K(\mu), \sigma(H_{U_h}(\mu)))\geq
C_\epsilon e^{-\epsilon \mu^{1/2}}, \quad h\in D\,,
\]
for all sufficiently large $\mu$.
 Then for any $z\in K(\mu)$ the operator
$(H_{U_h}(\mu)-z)^{-1}$ defines a bounded operator in
$L^2_{\mu^{1/2}\Phi_h}({U_h})$ and, for any $\epsilon>0$, there exists $C_{1,\epsilon}>0$ such that 
\[
\|(H_{U_h}(\mu)-z)^{-1}\|_{\mu^{1/2}\Phi_h}\leq
C_{1,\epsilon}e^{\epsilon \mu^{1/2}}, \quad h\in D\,,
\]
for all sufficiently large $\mu$.
\end{prop}

\begin{proof}
For every sufficiently small $\eta>0$, we take $C^\infty_b$-bounded families $\{\chi_{1,\eta}\in
C^\infty_c(U_h) : h\in D\}$ and $\{\chi^\prime_{1,\eta}\in C^\infty(U_h) : h\in D\}$ such that: 

(a) $0\leq \chi_{1,\eta}\leq 1$ for any $x\in U_h$, $\chi_{1,\eta}\equiv 1$ in a neighborhood
of $\{x\in {U_h}:{\Phi_h}(x)\leq 2\eta \}$, and ${\Phi_h}\leq 3\eta$ on $\supp
\chi_{1,\eta}$. 

(b) $\chi^\prime_{1,\eta}\geq 0$ and $(\chi_{1,\eta})^2 +
(\chi^\prime_{1,\eta})^2\equiv 1$;

(c) there exists a constant $C>0$ such that:
\begin{equation}\label{e:eta-nabla}
\eta (|\nabla \chi_{1,\eta}|+|\nabla \chi^\prime_{1,\eta}|)\leq C, \quad h\in D\,.
\end{equation}

Let us show the existence of such families. Since $V\in C^\infty_b(M)$, it is uniformly Lipschitz: there exists $L>0$ such that  
\begin{equation}\label{e:LipV}
|V(x)-V(y)|\leq Ld(x,y), \quad x,y\in M,
\end{equation}
where $d$ stands for the geodesic distance on $M$.  

For any sufficiently small $\eta>0$, take any functions $F_\eta, G_\eta\in C^\infty(\RR)$ such that: 

(1) $F_\eta(u)=1$ for $u\leq E_3+\frac{2L\eta}{E_3-E_2}$, $F_\eta(u)=0$ for $u\geq E_3+\frac{3L\eta}{E_3-E_2}$, and $0<F_\eta(u)<1$ otherwise;

(2) $G_{\eta}(u)\geq 0$ and $(F_{\eta}(u))^2 +
(G_{\eta}(u))^2=1$  for any $u\in \RR$;

(3) for any $u\in\RR$, we have $F_\eta^\prime(u)\leq c/\eta$ and $G_\eta^\prime(u)\leq c/\eta,$ where $c>0$ is independent of $\eta$. 

For any $\eta>0$ such that $E_3+\frac{3L\eta}{E_3-E_2}<E_0$, put 
\[
\chi_{1,\eta}(x)=F_\eta(V(x)), \quad \chi^\prime_{1,\eta}(x)=G_\eta(V(x)), \quad x\in U_h.
\]

Take an arbitrary $x\in {U_h}$ such that ${\Phi_h}(x)\leq 2\eta$ and show that $\chi_{1,\eta}(x)=1.$

If $x\in U_{E_3,h}$, then $V(x)\leq E_3$ and, therefore, $\chi_{1,\eta}(x)=1.$

It is clear that, for any $x\not \in U_{E_3,h}$, we have $\Phi_h(x)=d_{E_2}(x,\partial U_{E_3,h})$. Moreover, since $\partial U_{E_3,h}$ is compact, there is $y\in \partial U_{E_3,h}$ such that 
\begin{equation}\label{e:Phi1}
d_{E_2}(x,y)=\Phi_h(x)\leq 2\eta.
\end{equation}

On $U_h\setminus  U_{E_3,h}$, we have 
\[
(E_3-E_2)  g\leq g_{E_2}\leq  (E_0-E_2)  g.
\]
It follows that, for any $x_1,x_2\in U_h\setminus  U_{E_3,h}$, we have
\begin{equation}\label{e:ggM}
d_{E_2}(x_1,x_2) \geq (E_3-E_2) d(x_1,x_2),
\end{equation}
where $d$ is the distance defined by $g$. 

By \eqref{e:Phi1} and \eqref{e:ggM}, we have
\[
d(x,y) \leq \frac{2\eta}{E_3-E_2}.  
\]
It is clear that $V(y)=E_3$. Therefore, by \eqref{e:LipV}, 
\[
V(x)-V(y)=V(x)-E_3\leq\frac{2L\eta}{E_3-E_2}.  
\]
Hence $V(x)\leq E_3+\frac{2L\eta}{E_3-E_2}$ and $\chi_{1,\eta}(x)=1$. 

Now take any $x\in {U_h}$ such that $\chi_{1,\eta}(x)=0$. Then $V(x)>E_3+\frac{3L\eta}{E_3-E_2}$ and by \eqref{e:LipV}, for any $y\in \partial U_{E_3,h}$,
\[
d(x,y)\geq \frac{1}{L}|V(x)-V(y)|=\frac{1}{L}(V(x)-E_3)>\frac{3\eta}{E_3-E_2}.  
\]
By \eqref{e:ggM}, we infer that $d_{E_2}(x,y)>3\eta$ and therefore
$${\Phi_h}(x)=d_{E_2}(x,\partial U_{E_3,h})>3\eta.$$

Finally, we have
\[
\nabla \chi_{1,\eta}(x)=F^\prime_\eta(V(x))\nabla V(x), \quad x\in U_h,
\]
that immediately implies \eqref{e:eta-nabla} and completes the proof of the existence of the functions $\chi_{1,\eta}$ and $\chi^\prime_{1,\eta}$.

Now we use the standard localization formula
\begin{multline*}
\|\nabla_{\bf A} (e^{{\mu^{1/2}{\Phi_h}}}u)\|^2= \|\nabla_{\bf A} (\chi_{1,\eta} e^{{\mu^{1/2}{\Phi_h}}}u)\|^2
+\|\nabla_{\bf A}(\chi^\prime_{1,\eta} e^{{\mu^{1/2}{\Phi_h}}}u)\|^2\\ -\| |\nabla
\chi_{1,\eta}| e^{{\mu^{1/2}{\Phi_h}}} u\|^2-\| |\nabla
\chi^\prime_{1,\eta}| e^{{\mu^{1/2}{\Phi_h}}} u\|^2\,.
\end{multline*}

By (\ref{e:energy}), it follows that
\begin{multline}\label{e:energy1}
  \int_{U_h} |\nabla_{\bf A}(\chi^\prime_{1,\eta} e^{\mu^{1/2}{\Phi_h}}u)|^2\,dx \\ +\mu \int_{{U_h}} e^{2\mu^{1/2}{\Phi_h}}(V-|\nabla{\Phi_h}|^2)|\chi^\prime_{1,\eta}u|^2\,dx - {\rm Re}\, z\int_{{U_h}} e^{2\mu^{1/2}{\Phi_h}}
|\chi^\prime_{1,\eta} u|^2\,dx\\ -  \| |\nabla
\chi_{1,\eta}| e^{{\mu^{1/2}{\Phi_h}}}\chi^\prime_{1,\eta} u\|^2-  \| |\nabla
\chi^\prime_{1,\eta}| e^{{\mu^{1/2}{\Phi_h}}}\chi^\prime_{1,\eta} u\|^2 \\={\rm Re}\, \int_{{U_h}} e^{2\mu^{1/2}{\Phi_h}} (H(\mu)-z)u \bar u\,dx-  \int_{U_h} |\nabla_{\bf A}(\chi_{1,\eta} e^{\mu^{1/2}{\Phi_h}}u)|^2\,dx\\
-\mu   \int_{{U_h}} e^{2\mu^{1/2}{\Phi_h}}(V-|\nabla{\Phi_h}|^2)|\chi_{1,\eta}u|^2\,dx + {\rm Re}\, z\int_{{U_h}} e^{2\mu^{1/2}{\Phi_h}}|\chi_{1,\eta} u|^2\,dx\\
+  \| |\nabla \chi_{1,\eta}| e^{{\mu^{1/2}{\Phi_h}}} \chi_{1,\eta} u\|^2+  \| |\nabla
\chi^\prime_{1,\eta}| e^{{\mu^{1/2}{\Phi_h}}}\chi_{1,\eta}  u\|^2.
\end{multline}

Put $\eta=\alpha \mu^{-1/2}$ with sufficiently large $\alpha>0$, which will be chosen later.
Taking into account \eqref{e:eta-nabla}:
\[
|\nabla \chi_{1,\eta}|+|\nabla \chi^\prime_{1,\eta}|\leq
\frac{C}{\eta}=\frac{C}{\alpha} \mu^{1/2},
\]
we get the following estimate for
the right-hand side of (\ref{e:energy1})
\begin{multline*}
-\mu  \int_{{U_h}} e^{2\mu^{1/2}{\Phi_h}}(V-|\nabla{\Phi_h}|^2)|\chi_{1,\eta}u|^2\,dx + {\rm Re}\, z\int_{{U_h}} e^{2\mu^{1/2}{\Phi_h}}|\chi_{1,\eta} u|^2\,dx\\
+  \| |\nabla \chi_{1,\eta}| e^{{\mu^{1/2}{\Phi_h}}} \chi_{1,\eta} u\|^2+  \| |\nabla
\chi^\prime_{1,\eta}| e^{{\mu^{1/2}{\Phi_h}}}\chi_{1,\eta}  u\|^2
\\
 \leq C\mu  \|e^{{\mu^{1/2}{\Phi_h}}} \chi_{1,\eta}u\|^2\,.
\end{multline*}
On the other hand we have the following estimate for the left-hand side of (\ref{e:energy1}):
\begin{multline*}
 \int_{U_h} |\nabla_{\bf A}(\chi^\prime_{1,\eta} e^{\mu^{1/2}{\Phi_h}}u)|^2\,dx \\ +\mu  \int_{{U_h}} e^{2\mu^{1/2}{\Phi_h}}(V-|\nabla{\Phi_h}|^2)|\chi^\prime_{1,\eta}u|^2\,dx - {\rm Re}\, z\int_{{U_h}} e^{2\mu^{1/2}{\Phi_h}}
|\chi^\prime_{1,\eta} u|^2\,dx\\ -  \| |\nabla
\chi_{1,\eta}| e^{{\mu^{1/2}{\Phi_h}}}\chi^\prime_{1,\eta} u\|^2-  \| |\nabla
\chi^\prime_{1,\eta}| e^{{\mu^{1/2}{\Phi_h}}}\chi^\prime_{1,\eta} u\|^2 \\
\geq \int_{{U_h}} e^{{2\mu^{1/2}{\Phi_h}}}\, \left[\mu  (V-|\nabla{\Phi_h}|^2)-{\rm Re}\, z-\frac{C^2}{\alpha^2}\mu\right]\, |\chi^\prime_{1,\eta}
u(x)|^2\,dx\\
  \geq C\mu \|e^{{\mu^{1/2}{\Phi_h}}}
\chi^\prime_{1,\eta}u\|^2\,.
\end{multline*}
Here we use the fact that $|\nabla{\Phi_h}|^2<V-E_1$ and choose $\alpha$ to be large enough. 

Thus, from (\ref{e:energy1}), we get the estimate
\[
c\mu \|e^{{\mu^{1/2}{\Phi_h}}} u\|^2 \leq {\rm Re}\, \int_{{U_h}} e^{2\mu^{1/2}{\Phi_h}} (H(\mu)-z)u \bar u\,dx + C  \mu 
\|e^{{\mu^{1/2}{\Phi_h}}} \chi_{1,\eta}u\|^2\,.
\]
It remains to show that, for any $\epsilon>0$, there exists $C_\epsilon>0$ such that 
\begin{equation*}
\|e^{{\mu^{1/2}{\Phi_h}}}\chi_{1,\eta}u\| \leq C_\epsilon
e^{{\epsilon\mu^{1/2}}} \|e^{{\mu^{1/2}{\Phi_h}}}(H_{U_h}(\mu)-z)u\|\,,
\end{equation*}
or equivalently,
\begin{equation}\label{e:chi}
\|\chi_{1,\eta}(H_{U_h}(\mu)-z)^{-1}u\|_{\mu^{1/2}{\Phi_h}} \leq C_\epsilon
e^{{\epsilon\mu^{1/2}}} \|u\|_{\mu^{1/2}{\Phi_h}}, \quad u\in
L^2_{\mu^{1/2}{\Phi_h}}({U_h})\,,
\end{equation}
for any sufficiently large $\mu$.

For this, we choose a function $\chi_{2,\eta}\in C^\infty_c({U_h})$
such that $\chi_{2,\eta}\equiv 1$ in a neighborhood of $\{x\in
{U_h}:{\Phi_h}(x)\leq \eta\}$, ${\Phi_h}\leq 2\eta$ on $\supp \chi_{2,\eta}$.
In particular, $\chi_{1,\eta}\equiv 1$ on $\supp \chi_{2,\eta}$.
We can assume that there exists a constant $C$ such that for all
sufficiently small $\eta>0$
\begin{equation}\label{e:chi2}
\eta |\nabla \chi_{2,\eta}|+\eta^2|\Delta \chi_{2,\eta}|\leq C.
\end{equation}

Let $M_0=\{x\in {U_h}: {\Phi_h}(x)\geq 2\eta\}$. Then we have
\begin{multline*}
(H_{U_h}(\mu)-z)^{-1}u=(1-\chi_{2,\eta})(H_{M_0}(\mu)-z)^{-1}(1-\chi_{1,\eta})u+
(H_{U_h}(\mu)-z)^{-1}\chi_{1,\eta}u \\ + (H_{U_h}(\mu)-z)^{-1}\chi_{1,\eta}
[H_{U_h}(\mu),\chi_{2,\eta}] (H_{M_0}(\mu)-z)^{-1}(1-\chi_{1,\eta})u\,.
\end{multline*}
We consider three terms in the right hand side of the last
identity separately. For the first one we use
Proposition~\ref{p:1} and obtain
\begin{equation}\label{e:1}
\|\chi_{1,\eta}(1-\chi_{2,\eta})(H_{M_0}(\mu)-z)^{-1}
(1-\chi_{1,\eta})u\|_{\mu^{1/2}{\Phi_h}}\leq C\mu
\|u\|_{\mu^{1/2}{\Phi_h}}\,.
\end{equation}
For the second term, since ${\Phi_h}\leq 3\eta$ on $\supp
\chi_{1,\eta}$, we have
\[
\|\chi_{1,\eta}(H_{U_h}(\mu)-z)^{-1}\chi_{1,\eta}u\|_{\mu^{1/2}{\Phi_h}}
\leq e^{3\alpha} \|(H_{U_h}(\mu)-z)^{-1}\chi_{1,\eta} u\|\,.
\]
By the assumptions and the fact that ${\Phi_h}\geq 0$, it follows that
\[
\|(H_{U_h}(\mu)-z)^{-1}\chi_{1,\eta} u\|\leq C e^{\epsilon\mu^{1/2}}
\|\chi_{1,\eta} u\| \leq C_1 e^{{\epsilon\mu^{1/2}}} \|
u\|_{\mu^{1/2}{\Phi_h}}, \quad {\mu\gg 1}\,.
\]
So we get for the second term
\begin{equation}\label{e:2}
\|\chi_{1,\eta}(H_{U_h}(\mu)-z)^{-1}\chi_{1,\eta}u\|_{\mu^{1/2}{\Phi_h}}
\leq C_2 e^{{\epsilon\mu^{1/2}}} \| u\|_{{\Phi_h}/{\mu}}, \quad {\mu\gg 1}\,.
\end{equation}

For the third term we put $w=(H_{M_0}(\mu)-z)^{-1} (1-
\chi_{1,\eta})u$. By (\ref{e:2}), it follows that
\begin{multline*}
\|\chi_{1,\eta}(H_{U_h}(\mu)-z)^{-1}\chi_{1,\eta} [H_{U_h}(\mu),
\chi_{2,\eta}] w\|_{\mu^{1/2}{\Phi_h}}\\ \leq C_1
e^{{\epsilon\mu^{1/2}}}\|[H_{U_h}(\mu), \chi_{2,\eta}]
w\|_{\mu^{1/2}{\Phi_h}}, \quad {\mu\gg 1}\,.
\end{multline*}
Now we have
\[
[H_{U_h}(\mu), \chi_{2,\eta}]w=2 \, d \chi_{2,\eta}\cdot \nabla_{\bf A} w +  \Delta \chi_{2,\eta} w\,.
\]
Therefore, taking into account (\ref{e:chi2}), we get
\begin{align*}
\|[H_{U_h}(\mu), \chi_{2,\eta}] w\|^2_{\mu^{1/2}{\Phi_h}} & \leq C
(\mu^2\|\nabla_{\bf A} w\|^2_{\mu^{1/2}{\Phi_h}}+ \mu^4 \|w\|^2_{\mu^{1/2}{\Phi_h}})\\ & \leq C(\mu^2\|\nabla_{\bf A}(e^{{\mu^{1/2}{\Phi_h}}}w)\|^2+ \mu^4   \|w\|^2_{\mu^{1/2}{\Phi_h}})\,.
\end{align*}
By Proposition \ref{p:1} and Corollary~\ref{c:h1}, we have
\begin{align*}
\|[H_{U_h}(\mu), \chi_{2,\eta}] w\|^2_{\mu^{1/2}{\Phi_h}}  \leq & C(\mu^2
\|\nabla_{\bf A}(e^{{\mu^{1/2}{\Phi_h}}}(H_{M_0}(\mu)-z)^{-1} (1- \chi_{1,\eta})u)\|^2 \\ &
+  \mu^4 \|(H_{M_0}(\mu)-z)^{-1} (1-
\chi_{1,\eta})u\|^2_{\mu^{1/2}{\Phi_h}}) \\ \leq & C \mu^2\|(1-
\chi_{1,\eta})u\|^2_{\mu^{1/2}{\Phi_h}} \leq C \mu^2
\|u\|^2_{\mu^{1/2}{\Phi_h}}\,.
\end{align*}
So we get for the third term
\begin{multline}\label{e:3}
\|\chi_{1,\eta}(H_{U_h}(\mu)-z)^{-1}\chi_{1,\eta} [H_{U_h}(\mu),
\chi_{2,\eta}] (H_{M_0}(\mu)-z)^{-1} (1-
\chi_{1,\eta})u\|_{\mu^{1/2}{\Phi_h}}\\ \leq C_{3,\epsilon}
e^{{\epsilon\mu^{1/2}}} \|u\|_{\mu^{1/2}{\Phi_h}}, \quad {\mu\gg 1}\,.
\end{multline}
Now (\ref{e:chi}) follows by adding the estimates (\ref{e:1}),
(\ref{e:2}) and (\ref{e:3}).
\end{proof}

\begin{cor}\label{c:h2}
Under the assumptions of Proposition~\ref{p:10}, for any
$\epsilon>0$\,, there exists $C_{2,\epsilon}>0$ such that  
\begin{multline*}
\|\nabla_{\bf A}(e^{{\mu^{1/2}{\Phi_h}}}(H_{U_h}(\mu)-z)^{-1}v)\| \leq
C_{2,\epsilon}e^{\epsilon\mu^{1/2}}
\|v\|_{\mu^{1/2}{\Phi_h}},\\ v\in L^2_{\mu^{1/2}{\Phi_h}}({U_h}), \quad h\in D,
\end{multline*}
for any sufficiently large $\mu$.
\end{cor}

\begin{proof}
By Lemma \ref{l:1.1}, for any $\mu>0$, $z\in\CC$ and
$v\in L^2_{\mu^{1/2}{\Phi_h}}({U_h})$ one has
\begin{multline*}
\|\nabla_{\bf A}(e^{{\mu^{1/2}{\Phi_h}}}(H_{U_h}(\mu)-z)^{-1}v)\|^2  = {\rm Re}\, \int_{{U_h}} e^{2\mu^{1/2}\Phi_h} v \cdot \overline{(H_{U_h}(\mu)-z)^{-1}v}\,dx\\ - \int_{U_h} e^{2\mu^{1/2}\Phi_h}(\mu(V-|\nabla\Phi_h|^2) - {\rm Re}\, z)|(H_{U_h}(\mu)-z)^{-1}v|^2\,dx.
\end{multline*}
\end{proof}

\subsection{Proof of Theorem~\ref{t:D}}
 
Theorem~\ref{t:D} follows immediately from the following

\begin{prop}\label{p:D}
Let Assumption \ref{a1} hold with the fixed $E_0$ and $E_1\in (0,E_0)$.
Assume that $K(\mu)$ is a bounded subset in $\CC$ such that
$K(\mu)\subset \{z\in \CC : {\rm Re}\,z < E_1\mu \}$, and, for any $\epsilon >0$\,, there exists $C_\epsilon>0$ such that  
\[
{\rm dist}\,(K(\mu), \sigma(H_{U_h}(\mu)))\geq
\frac{1}{C_\epsilon}e^{-\epsilon \mu^{1/2}}, \quad h\in D\,
\]
for any sufficiently large $\mu$.
Then $$K(\mu)\cap \sigma(H(\mu)) = \emptyset, \quad \mu\gg 1\,.
$$

\end{prop}

\begin{proof}
Without loss of generality, we can assume that \eqref{e:epsilon1} holds.
Take any $\eta>0$ such that $E_1+3\eta<E_0$. Let
\[
M_0=M\setminus \bigcup_{h\in D} 
U_{E_1+\eta,h}= \{x\in { M}\,:\,   V(x) \geq
 E_1+\eta \}\,.
\]
Take any function $\chi \in C^\infty(\RR)$ such that $\chi(u)=0$ for $u>E_1+2\eta$, $\chi(u)=1$ for $u<E_1+\eta$ and $0\leq \chi(u)\leq 1$ for any $u\in\RR$. It is easy to see that the function $\phi(x)=\chi(V(x))$ is a $C^\infty_b$-function on $M$, which is of the form
\[ 
\phi=\sum_{h\in D}\phi_h\,,
\]
where $\phi_h\in C^\infty_c({ M})$ such that $\supp
\phi_h \subset U_{E_1+2\eta,h}$, $\phi_h\equiv 1$ on
$U_{E_1+\eta,h}$. Moreover, the family $\{\phi_h : h\in D\}$ is bounded in $C^\infty_b(M)$. Let
\[
\phi_0=1-\sum_{h\in D}\phi_h=(1-\chi)\circ V\,.
\]
Then $\supp \phi_0\subset M_0$. 

Let $\chi_1 \in C^\infty(\RR)$ be such that $\chi_1(u)=0$ for $u>E_1+3\eta$, $\chi_1(u)=1$ for $u<E_1+5/2\eta$ and $0\leq \chi_1(u)\leq 1$ for any $u\in\RR$. It is easy to see that the function $\psi(x)=\chi_1(V(x))$ is a $C^\infty_b$-function on $M$, which is of the form
\[ 
\psi=\sum_{h\in D}\psi_h\,,
\]
where   $\psi_h\in C^\infty_c({ M})$,
$h\in D$, such that $\supp \psi_h \subset
U_{E_1+3\eta,h}$, $\psi_h\equiv 1$ in 
$U_{E_1+5/2\eta,h}$. We can assume that there exists a constant $C$ such that for all
sufficiently small $\eta>0$
\begin{equation}\label{e:psi-gamma}
\eta |\nabla \psi_h|+\eta^2|\Delta \psi_h|\leq C, \quad h\in D.
\end{equation}

Take any function $\chi_2 \in C^\infty(\RR)$ be such that $\chi_2(u)=0$ for $u<E_1+\eta$, $\chi_2(u)=1$ for $u>E_1+2\eta$ and $0\leq \chi(u)\leq 1$ for any $u\in\RR$. Set $\psi_0=\chi_2\circ V$. Then $\psi_0\in C^\infty_b(M)$ is such that $\supp \psi_0 \subset M_0$,
$\psi_0\equiv 1$ in a neighborhood of $M\setminus
\cup_{h\in D} U_{E_1+2\eta,h}$. In particular, we have $\phi_0\psi_0=\phi_0$ and
$\phi_h\psi_h=\phi_h$, for $h\in D$.

For any $\mu>0$ large enough and any $z\in K(\mu)$, define a bounded
operator $R^\mu(z)$ in $L^2( M)$ as
\begin{equation}\label{Rmu}
R^\mu(z)=\sum_{h\in D} \psi_h
(H_{U_h}(\mu)-z)^{-1}\phi_h + \psi_0
(H_{M_0}(\mu)-z)^{-1}\phi_0\,.
\end{equation}
Then
\[
(H(\mu)-z)R^\mu(z)=I+K^\mu(z)\,,
\]
where
\[
K^\mu(z)=\sum_{h\in D} [H(\mu),\psi_h]
(H_{U_h}(\mu)-z)^{-1}\phi_h + [H(\mu),\psi_0]
(H_{M_0}(\mu)-z)^{-1}\phi_0\,.
\]

\begin{lem}\label{l:kh}
There exist $C,c >0$ such that, for any $\mu >0$ large enough and
$z\in K(\mu)$,  the operator $K^\mu(z)$ defines a bounded operator in
$L^2( M)$ with the norm estimate
\[
\|K^\mu(z)\|\leq C e^{-c\mu^{1/2}}\,.
\]
\end{lem}

\begin{proof}
For any $h\in D$, consider a weight function $\Phi_h\in
\cW_{E_1}(U_h)$ given by
$\Phi_h(x)=d_{E_1+\eta}(x,U_{E_1+2\eta,h})$.
By construction, $\Phi_h(x)\equiv 0$ on $\supp \phi_h$. We claim that there exists $c_1>0$ such that, for any $h\in D$, $\Phi_h(x)\geq c_1>0$ on $\supp d\psi_h$. 

Indeed, we know that $\supp d\psi_h \subset
U_{E_1+3\eta,h}\setminus  U_{E_1+5/2\eta,h}$.

It is clear that, for any $x\not \in \supp U_{E_1+2\eta,h}$, 
$$\Phi_h(x)=d_{E_1+\eta}(x,\partial U_{E_1+2\eta,h}).$$

By \eqref{e:LipV}, for any $x\in \supp d\psi_h$ and $y\in \partial U_{E_1+2\eta,h}$, we have 
\[
d(x,y)>\frac{\eta}{2L}. 
\]
By \eqref{e:ggM}, for any $x\in \supp d\psi_h$ and $y\in \partial U_{E_1+2\eta,h}$,
\[
d_{E_1+\eta}(x,y) \geq \eta d(x,y) >\frac{\eta^2}{2L}
\] 
and 
\[
\Phi_h(x)=d_{E_1+\eta}(x,\partial U_{E_1+2\eta,h})> \frac{\eta^2}{2L}=:c_1,
\]
as desired.

For any $w\in {\rm
Dom}\,H(\mu)$, we have
\[
[H(\mu), \psi_h]w=2 \, d \psi_h \cdot \nabla_{\bf A} w +  \Delta
\psi_h\, w\,.
\]
This  implies the estimate
\[
\|[H(\mu), \psi_h]w\|^2_{\mu^{1/2}\Phi_h}   \leq C  (\|\nabla_{\bf A} w\|^2_{\mu^{1/2}\Phi_h}+\|w\|^2_{\mu^{1/2}\Phi_h})\,.
\]
Therefore, for any $u\in L^2( M)$, we obtain
\begin{multline*}
\|[H(\mu), \psi_h] (H_{U_h}(\mu)-z)^{-1}\phi_h u\|^2_{L^2( M)}\\
\begin{aligned}
 = &
\|[H(\mu), \psi_h] (H_{U_h}(\mu)-z)^{-1}\phi_h u\|^2_{L^2(U_h)}\\  \leq &
e^{-c_1\mu^{1/2}} \|[H(\mu), \psi_h] (H_{U_h}(\mu)-z)^{-1}\phi_h
u\|^2_{\mu^{1/2}\Phi_h}\\ \leq & C e^{-c_1\mu^{1/2}}  (
\|\nabla_{\bf A}((H_{U_h}(\mu)-z)^{-1}\phi_h u)\|^2_{\mu^{1/2}\Phi_h} \\ & +  
\|(H_{U_h}(\mu)-z)^{-1}\phi_h u\|^2_{\mu^{1/2}\Phi_h})\,.
\end{aligned}
\end{multline*}
It follows from Proposition \ref{p:10} and Corollary \ref{c:h2} that, for any $\epsilon>0$, there exists $C_\epsilon>0$ such that, for any sufficiently large $\mu$,
\begin{align*}
\|[H(\mu), \psi_h] (H_{U_h}(\mu)-z)^{-1}\phi_h u\|_{L^2( M)} \leq &
C_\epsilon e^{-(c_1-\epsilon)\mu^{1/2}} \|\phi_h
u\|_{\mu^{1/2}\Phi_h}\\ = & C_\epsilon
e^{-(c_1-\epsilon)\mu^{1/2}} \|\phi_h u\|_{L^2(U_h)}\\ \leq &
C_\epsilon e^{-(c_1-\epsilon)\mu^{1/2}} \| u\|_{L^2( M)}\,.
\end{align*}
Here we used the facts that, for any $h\in D$, $\Phi_h(x)\equiv 0$ on $\supp \phi_h$ and $\Phi_h(x)\geq c_1>0$ on $\supp d\psi_h$. 

Similarly, using Proposition \ref{p:1} and Corollary \ref{c:h1}, one can get
\[
\|[H(\mu), \psi_0] (H_{M_0}(\mu)-z)^{-1}\phi_0 u\|_{L^2( M)} \leq C_0
e^{-c_0\mu^{1/2}} \| u\|_{L^2( M)}\,.
\]
Taking into account that the supports of $\phi_h$ with $h\in D$ are disjoint, we get
\begin{align*}
\|K^\mu(z)u\| & \leq C e^{-c\mu^{1/2}} (
\sum_{h\in D} \|\phi_h u\| + \|\phi_0 u\|)\\ &
\leq C_1 e^{-c\mu^{1/2}} \|u\|\,.
\end{align*}
This  completes the proof.
\end{proof}

It follows from Lemma~\ref{l:kh} that, for all sufficiently large
$\mu>0$ and $z\in K(\mu)$, the operator $I+K^\mu(z)$ is invertible in
$L^2(M)$. Then the operator $H(\mu)-z$ is invertible in $L^2( M)$
with
\begin{equation}\label{H-R}
(H(\mu)-z)^{-1}= R^\mu(z)(I+K^\mu(z))^{-1}\,,
\end{equation}
and $K(\mu)\cap \sigma(H(\mu))=\emptyset$ as desired.
\end{proof}

\section{Proof of Theorem \ref{t:equiv}} \label{s:equiv}

The main goal of this section is the proof of Theorem \ref{t:equiv}. First, we will prove two auxiliary results. 

\subsection{Rough estimate for eigenspace dimensions.}
Assume that Assumption \ref{a1} holds with the fixed $E_0$ and $E_1$ satisfies \eqref{e:epsilon1}. For any relative compact open domain $W\subset  M$ with a regular boundary and for any $\lambda\in \RR$ and $\mu>0$, let $E_{H_W(\mu)}(\lambda) =  \chi_{(-\infty, \lambda]}(H_W(\mu))$ denote the spectral projection of the operator $H_W(\mu)$ corresponding to the semi-axis $(-\infty, \lambda]$. The operator $E_{H_U(\mu)}(\lambda)$ is a bounded operator in $L^2(U)$ of the form
\[
E_{H_U(\mu)}(\lambda)=\bigoplus_{h\in D} E_{H_{U_h}(\mu)} (\lambda).
\]

\begin{lem}\label{l:Weyl}
There exists $C>0$ and $\mu_0>0$ such that 
\[
\sup_{h\in D}\dim \operatorname{Im} E_{H_{U_h}(\mu)} (E_1\mu)<C\mu^d, \quad \mu>\mu_0.
\]
\end{lem}

\begin{proof}
Let $r_0>0$ be the injectivity radius of $M$. Fix $r<r_0$. Then each open ball $B(x,r)\subset M$ is a relative compact open domain with smooth boundary. The exponential map $\exp^M_x : B(0,r)\subset T_xM\to B(x,r)\subset M$ along with a choice of an orthonormal frame in $T_xM$ for any $x\in M$ defines a normal coordinate system $h_x : B(0,r)\subset \RR^n\to B(x,r)\subset M$, $x\in M$. By bounded geometry conditions, the operators $H_{B(x,r)}(\mu)$, written in the normal coordinates, define a $C^\infty$-bounded family $H_{x}(\mu)$ of second order differential operators on $B(0,r)\subset \RR^n$. We consider them as unbounded operator in $L^2(B(0,r), dv_x)$, where $dv_x$ is the Riemannian volume form on $M$, written in the normal coordinates. By bounded geometry conditions, we know that $dv_x$ is $C^\infty$-bounded family of volume forms on $B(0,r)\subset \RR^n$. By the min-max principle, it is easy to see that, for any $\lambda\in \RR$, there exists $C>0$ such that for any $x\in M$ we have  
\[
\dim \operatorname{Im} E_{H_{B(x,r)}(\mu)} (\lambda)<\dim \operatorname{Im} E_{H_{B(x,r)}(0)} (\lambda)<C\lambda^d, \quad \mu>0, \quad \lambda>0.
\]

Next, we claim that there exists $N\in \mathbb N$ such that each $U_h$, $h\in D$, is covered by at most $N$ balls of the form $B(x,r)$. Indeed, by \eqref{e:tr0}, there exists $R>0$ such that each $U_h$ is contained in a ball $B(x_h,R)$ of radius $R$ centered at some $x_h\in M$. Consider a maximal (with respect to inclusion) set $y_{1,h},y_{2,h},\ldots, y_{\nu_h,h}$ of points in $B(x_h,R)$ such that $d(y_{j,h},y_{k,h})\geq r$ for any $j,k=1,2,\ldots,{\nu_h},$ $j\neq k$. It is clear that 
\[
B(x_h,R)\subset \bigcup_{j=1}^{\nu_h} B(y_{j,h},r),
\]  
and therefore  $U_h$ is covered by the balls $B(y_{j,h},r)$, $j=1,\ldots,{\nu_h}$.
 
On the other hand, each ball $B(y_{j,h},r/2)$ is contained in $B(x_h,R+r/2)$ and $B(y_{j,h},r/2)\cap B(y_{k,h},r/2)=\emptyset$ for any $j,k=1,2,\ldots,{\nu_h},$ $j\neq k$. Therefore, we have
\[
\sum_{j=1}^{\nu_h} {\rm vol} (B(y_{j,h},r/2)) \leq {\rm vol} (B(x_h,R+r/2)).
\] 
By bounded geometry conditions, we have 
\[
v_1(\rho) \leq {\rm vol} (B(x,\rho))\leq v_2(\rho), \quad \rho\in \RR_+. 
\]
We infer that 
\[
{\nu_h}\leq N:=\left[\frac{v_2(R+r/2)}{v_1(r/2)}\right]+1, \quad h\in D.
\]
Using this fact, we can proceed as in the proof of Lemma 4.2 in \cite{HM96}. We get 
\[
\dim \operatorname{Im} E_{H_{U_h}(\mu)} (E_1\mu)\leq \sum_{j=1}^{\nu_h} \dim \operatorname{Im} E_{H_{B(y_{j,h},r)}(\mu)} (E_1\mu+C)\leq N  C_1\mu^d, \mu>\mu_0,
\]
that completes the proof.
\end{proof}

\subsection{Decay of eigenfunctions}
As above, we assume that Assumption \ref{a1} holds with the fixed $E_0$ and $E_1$ satisfies \eqref{e:epsilon1}.

\begin{prop}\label{p:10a}
Assume that 
\[
H_{U_h}(\mu)u_\mu=\lambda(\mu)u_\mu
\]
with some $u_\mu\in \Dom (H_{U_h}(\mu))$, $\|u_\mu\|=1$ and $\lambda(\mu)\leq E_1\mu$. Then, for any $E_2\in (E_1,E_0)$, there exists $c>0$ such that, if $\mu>0$ is large enough, then
\[
\int_{U_h\setminus U_{E_2,h}}|u_\mu|^2dx\leq Ce^{-c \mu^{1/2}}, \quad h\in D\,.
\]
\end{prop}

\begin{proof}
For any $h\in D$, consider a weight function $\Phi_h\in
\cW_{E_1}(U_h)$ given by $\Phi_h(x)= d_{E_2}(x,U_{E_3,h})$ with some $E_3\in (E_1,E_2)$.

By Lemma \ref{l:1.1}, we have 
\[
\int_{U_h} e^{2\mu^{1/2}\Phi_h}(\mu(V-|\nabla\Phi_h|^2) - \lambda(\mu))|u_\mu|^2\,dx\leq 0,
\]
or
\begin{multline}\label{e:1.1}
\int_{U_h\setminus U_{E_3,h}} e^{2\mu^{1/2}\Phi_h}(\mu(V-|\nabla\Phi_h|^2) - \lambda(\mu))|u_\mu|^2\,dx \\ \leq -\int_{U_{E_3,h}} e^{2\mu^{1/2}\Phi_h}(\mu(V-|\nabla\Phi_h|^2) - \lambda(\mu))|u_\mu|^2\,dx.
\end{multline}
For the right-hand side of \eqref{e:1.1}, we have the estimate
\[
\left|\int_{U_{E_3,h}} e^{2\mu^{1/2}\Phi_h}(\mu(V-|\nabla\Phi_h|^2) - \lambda(\mu))|u_\mu|^2\,dx\right|
  \leq C\mu \int_{U_{E_3,h}} |u_\mu|^2\,dx,
\]
where 
\[
C=\sup_{U_{E_3,h}}(|\nabla\Phi_h|^2+V +E_1).
\]
On the other hand, on $U_h\setminus U_{E_3,h}$, we have 
\[
\mu(V-|\nabla\Phi_h|^2) - \lambda(\mu)\geq  \mu E_2- \lambda(\mu)\geq (E_2-E_1)\mu.
\]
Therefore, for the left-hand side of \eqref{e:1.1}, we have
\begin{multline*}
\int_{U_h\setminus U_{E_3,h}} e^{2\mu^{1/2}\Phi_h}(\mu(V-|\nabla\Phi_h|^2) - \lambda(\mu))|u_\mu|^2\,dx\\ \geq (E_2-E_1)\mu \int_{U_h\setminus U_{E_3,h}} e^{2\mu^{1/2}\Phi_h}|u_\mu|^2\,dx.
\end{multline*}
Thus, by \eqref{e:1.1}, we get 
\[
\int_{U_h\setminus U_{E_3,h}} e^{2\mu^{1/2}\Phi_h} |u_\mu|^2\,dx \leq C  \int_{U_h} |u_\mu|^2\,dx=C.
\]
As in the proof of Proposition \ref{p:10}, one can show that $\Phi_h>c/2$ on $U_h\setminus U_{E_2,h}$ with some $c>0$. Therefore, we have
\begin{multline*}
\int_{U_h\setminus U_{E_2,h}}|u_\mu|^2dx\leq Ce^{-c \mu^{1/2}}\int_{U_h\setminus U_{E_2,h}} e^{2\mu^{1/2}\Phi_h} |u_\mu|^2\,dx\\
\leq Ce^{-c \mu^{1/2}}\int_{U_h\setminus U_{E_3,h}} e^{2\mu^{1/2}\Phi_h} |u_\mu|^2\,dx \leq Ce^{-c \mu^{1/2}},
\end{multline*}
that completes the proof. 
\end{proof}

\subsection{Proof of Theorem \ref{t:equiv}}
Now we complete the proof of Theorem \ref{t:equiv}. So we assume that Assumption \ref{a1} holds with the fixed $E_0$, $E_1$ satisfies \eqref{e:epsilon1}, and $\lambda(\mu) \in (0,E_1\mu)$ satisfies the condition: 
for any $\epsilon >0$\,, there exists $C_\epsilon>0$ such that 
\[
{\rm dist}\,(\lambda(\mu), \sigma(H_{U}(\mu)))\geq
\frac{1}{C_\epsilon}e^{-\epsilon \mu^{1/2}},
\]
  for all sufficiently large $\mu$.
By Proposition~\ref{p:D}, $\lambda(\mu)$ is not in the spectrum of $H(\mu)$  for sufficiently large $\mu>0$. It is a quite standard result that the spectral projection $E_{H(\mu)}(\lambda)$ is in $C^*(M)$ when $\lambda$ is in a spectral gap of the operator (see, for instance, \cite{kms,LT22}). The fact that the projection $P_{\mathcal H_{U,\mu} (\lambda)}$ is in $C^*(M)$ will be proved in Section~\ref{s:PHU} below in a more general setting.

Now we claim that for $\lambda=\lambda(\mu)$ satisfying the assumptions of the theorem, we have
\begin{equation}\label{e:E-pH}
E_{H(\mu)}(\lambda)  - P_{\mathcal H_{U,\mu}(\lambda)}=\mathcal O(e^{-c\mu^{1/2}}), \quad \mu\to +\infty. 
\end{equation}
By a well-known fact, this immediately implies Murray-von Neumann equivalence of the projections $E_{H(\mu)}(\lambda)$ and $P_{\mathcal H_{U,\mu}(\lambda)}$ for sufficiently large $\mu>0$.

By the Riesz formula one has, $$ 
E_{H(\mu)}(\lambda) = \frac{i}{2\pi} \oint_\Gamma (H(\mu)-z)^{-1} 
dz, $$
where $\Gamma$ is a contour intersecting the real axis at $\lambda$ and at some large negative number not in the spectrum of $H(\mu)$. Similarly, 
$$ 
E_{H_{U_h}(\mu)} (\lambda) =
\frac{i}{2\pi} \oint_\Gamma (H_{U_h}(\mu)-z)^{-1} 
dz. $$
We will use notation introduced in the proof of Proposition \ref{p:D}, in particular, the functions $\psi_h, h\in D,$ and $\psi_0$ and the operators $R^\mu(z)$ and $K^\mu(z)$.  By \eqref{H-R} and Lemma \ref{l:kh}, we infer that
$$ 
E_{H(\mu)}(\lambda)  =
\frac{i}{2\pi} \oint_\Gamma R^\mu(z) dz+\mathcal O(e^{-c\mu^{1/2}}), \quad \mu\to +\infty, $$
in the operator norm in $L^2(M)$. 

By \eqref{Rmu}, we have  
\begin{align*}
\frac{i}{2\pi } \oint_\Gamma R^\mu(z)dz=& \sum_{h\in D} \psi_h
\frac{i}{2\pi } \oint_\Gamma (H_{U_h}(\mu)-z)^{-1}dz \phi_h\\ & + \psi_0
\frac{i}{2\pi } \oint_\Gamma (H_{M_0}(\mu)-z)^{-1}dz \phi_0\,. 
\end{align*}
It is clear that 
\[
\frac{i}{2\pi} \oint_\Gamma (H_{M_0}(\mu)-z)^{-1}dz=0.
\]
We get 
\begin{equation}\label{e:E=ED}
E_{H(\mu)}(\lambda) =
\sum_{h\in D} \psi_h
E_{H_{U_h}(\mu)} (\lambda) \phi_h+\mathcal O(e^{-c\mu^{1/2}}), \quad \mu\to +\infty. 
\end{equation}

Now we show that 
\begin{equation}\label{e:pH=E}
P_{\mathcal H_{U,\mu}(\lambda)}= \sum_{h\in D} \phi_h E_{H_{U_h}(\mu)} (\lambda) \phi_h +\mathcal O(e^{-c\mu^{1/2}}), \quad \mu\to +\infty.
\end{equation}

First, observe that 
\begin{multline}\label{e:sum}
\|P_{\mathcal H_{U,\mu}(\lambda)}-\sum_{h\in D} \phi_h E_{H_{U_h}(\mu)} (\lambda) \phi_h\|\\ =\sup_{h\in D}\|P_{\mathcal H_{U_h,\mu}(\lambda)}- \phi_h E_{H_{U_h}(\mu)} (\lambda) \phi_h\|_{L^2(U_h)}.
\end{multline}
Let $u_{1,h,\mu},\ldots,u_{N_{h,\mu}(\lambda),h,\mu}$ be an orthonormal base in the image of the projection $E_{H_{U_h}(\mu)} (\lambda)$, $N_{h,\mu}(\lambda)=\dim E_{H_{U_h}(\mu)} (\lambda)$. Thus, the Schwartz kernel of $E_{H_{U_h}(\mu)} (\lambda)$ is given by
\[
e_{H_{U_h}(\mu)}(\lambda,x,y)=\sum_{j=1}^{N_{h,\mu}(\lambda)} u_{j,h,\mu}(x)\overline{u_{j,h,\mu}(y)}, \quad x,y\in U_h. 
\]
The Schwartz kernel of $ \phi_h E_{H_{U_h}(\mu)} (\lambda) \phi_h$ is given by
\[
\phi_h(x) e_{H_{U_h}(\mu)}(\lambda,x,y)\phi_h(y)=\sum_{j=1}^{N_{h,\mu}(\lambda)} \phi_h(x) u_{j,h,\mu}(x)\phi_h(y) \overline{u_{j,h,\mu}(y)}.
\]

The set $\phi_h u_{1,h,\mu},\ldots, \phi_h u_{N_{h,\mu}(\lambda),h,\mu}$ is a base in $\mathcal H_{U_h,\mu} (\lambda)$, the image of $\phi_h E_{H_{U_h}(\mu)} (\lambda)$. Thus, the Schwartz kernel of $P_{\mathcal H_{U_h,\mu} (\lambda)}$ is given by 
\[
p_{\mathcal H_{U_h,\mu} (\lambda)}(x,y)=\sum_{j,k=1}^{N_{h,\mu}(\lambda)} G_{jk,h,\mu} \phi_h(x) u_{j,h,\mu}(x)\phi_h(y) \overline{u_{k,h,\mu}(y)}, \quad x,y\in U_h,
\]
where $\{G_{jk,h,\mu}\}$ is the inverse of the matrix 
\[
g_{jk,h,\mu}=(\phi_h u_{j,h,\mu},\phi_h u_{k,h,\mu})_{L^2(U_h)}=\int_{U_h} |\phi_h(x)|^2 u_{j,h,\mu}(x)\overline{u_{k,h,\mu}(x)}dx.
\] 
Since $\phi_h\equiv 1$ on $U_{E_1+\eta,h}$, by Proposition \ref{p:10a}, we infer that 
\[
g_{jk,h,\mu}=\int_{U_h} u_{j,h,\mu}(x)\overline{u_{k,h,\mu}(x)}dx+\mathcal O(e^{-c\mu^{1/2}})=\delta_{jk}+\mathcal O(e^{-c\mu^{1/2}}).
\]
It follows that 
\[
G_{jk,h,\mu}=\delta_{jk}+\mathcal O(e^{-c\mu^{1/2}})
\]
and 
\begin{multline*}
p_{\mathcal H_{U_h,\mu} (\lambda)}(x,y)=\sum_{j=1}^{N_{h,\mu}(\lambda)} \phi_h(x) u_{j,h,\mu}(x)\phi_h(y) \overline{u_{j,h,\mu}(y)}(1+\mathcal O(e^{-c\mu^{1/2}})), \\ x,y\in U_h.
\end{multline*}

The operator $P_{\mathcal H_{U_h,\mu}(\lambda)}- \phi_h E_{H_{U_h}(\mu)}(\lambda) \phi_h$ is a Hilbert-Schmidt operator in $L^2(U_h)$ and its operator norm can be estimated by the Hilbert-Schmidt norm:
\begin{multline*}
\|P_{\mathcal H_{U_h,\mu}(\lambda)}- \phi_h E_{H_{U_h}(\mu)}(\lambda) \phi_h\|^2_{L^2(U_h)}\\ 
\leq \int_{U_h}\int_{U_h} \left|p_{\mathcal H_{U_h,\mu}(\lambda)}(x,y) - \sum_{j=1}^{N_{h,\mu}(\lambda)} \phi_h(x) u_{j,h,\mu}(x)\phi_h(y) \overline{u_{j,h,\mu}(y)}\right|^2dx\, dy \\
\leq C^2e^{-2c\mu^{1/2}} \int_{U_h}\int_{U_h}\left( \sum_{j=1}^{N_{h,\mu}(\lambda)} |\phi_h(x) u_{j,h,\mu}(x)| |{\phi_h(y) u_{j,h,\mu}(y)}|\right)^2 dx\, dy.
\end{multline*} 
Using the inequality $\frac{\sum_{j=1}^Na_j}{N}\leq \left(\frac{\sum_{j=1}^Na^2_j}{N}\right)^{1/2}$, we get
\begin{multline*}
\int_{U_h}\int_{U_h}\left( \sum_{j=1}^{N_{h,\mu}(\lambda)} |\phi_h(x) u_{j,h,\mu}(x)| |{\phi_h(y) u_{j,h,\mu}(y)}|\right)^2 dx\, dy\\ \leq N_{h,\mu}(\lambda) \int_{U_h}\int_{U_h}  \sum_{j=1}^{N_{h,\mu}(\lambda)} |\phi_h(x) u_{j,h,\mu}(x)|^2 |{\phi_h(y) u_{j,h,\mu}(y)}|^2 dx\, dy \\ = N_{h,\mu}(\lambda) \sum_{j=1}^{N_{h,\mu}(\lambda)} \|\phi_h u_{j,h,\mu}\|^2_{L^2(U_h)}<N_{h,\mu}(\lambda)^2, \quad h\in D. 
\end{multline*} 
It follows that 
\[
\|P_{\mathcal H_{U_h,\mu}(\lambda)}- \phi_h E_{H_{U_h}(\mu)}(\lambda) \phi_h\|_{L^2(U_h)}\leq CN_{h,\mu}(\lambda) e^{-c\mu^{1/2}}, \quad h\in D.
\]
Since $\lambda \in (0,E_1\mu)$, by \eqref{e:sum} and Lemma \ref{l:Weyl}, this proves \eqref{e:pH=E}.

Now we are ready to prove \eqref{e:E-pH}. By \eqref{e:E=ED} and \eqref{e:pH=E}, we have 
$$ 
E_{H(\mu)}(\lambda)  - P_{\mathcal H_{U,\mu}(\lambda)}=\sum_{h\in D} (\psi_h-\phi_h)
E_{H_{U_h}(\mu)} (\lambda) \phi_h + \mathcal O(e^{-c\mu^{1/2}}). 
$$
It remains to show that the first term in the right-hand side of the last identity is exponentially small.

Recall that $\supp \phi_h \subset U_{E_1+2\eta,h}$ and $\supp (\psi_h-\phi_h) \subset U_{E_1+3\eta,h} \setminus U_{E_1+\eta,h}$.
For any $h\in D$, consider a weight function $\Phi_h\in\cW_{E_1}(U_h)$ given by
$\Phi_h(x)=d_{E_1+\eta/2}(x,U_{E_1+\eta,h})$.
By construction, $\Phi_h(x)\equiv 0$ on $\supp \phi_h$. As above, one can show that there exists $c_1>0$ such that, for any $h\in D$, $\Phi_h(x)\geq c_1>0$ on $\supp (\psi_h-\phi_h)$. 

Recall that $\Gamma$ is a contour intersecting the real axis at $\lambda$ and at some large negative number not in the spectrum of $H(\mu)$. By Proposition~\ref{p:10}, for any $\mu>0$ is large enough and $z\in \Gamma$, the operator $(H_{U_h}(\mu)-z)^{-1}, h\in D,$ defines a bounded operator in
$L^2_{\mu^{1/2}\Phi_h}({U_h})$ and, for any $\epsilon>0$, there exists $C_{1,\epsilon}>0$ such that
\[
\|(H_{U_h}(\mu)-z)^{-1}\|_{\mu^{1/2}\Phi_h}\leq
C_{1,\epsilon}e^{\epsilon \mu^{1/2}}, \quad h\in D, z\in \Gamma, \mu\gg 1.
\]
 Therefore, for any $u\in L^2( M)$, we obtain
\begin{multline*}
\|(\psi_h-\phi_h) (H_{U_h}(\mu)-z)^{-1}\phi_h u\|^2_{L^2( M)}\\
\begin{aligned}
 = &
\|(\psi_h-\phi_h) (H_{U_h}(\mu)-z)^{-1}\phi_h u\|^2_{L^2(U_h)}\\  \leq &
e^{-c_1\mu^{1/2}} \|(\psi_h-\phi_h) (H_{U_h}(\mu)-z)^{-1}\phi_h
u\|^2_{\mu^{1/2}\Phi_h}\\ \leq & C e^{-c_1\mu^{1/2}}   
\|(H_{U_h}(\mu)-z)^{-1}\phi_h u\|^2_{\mu^{1/2}\Phi_h} \,.
\end{aligned}
\end{multline*}
It follows from Proposition~\ref{p:10} that, for any $\epsilon>0$,  there exists $C_{\epsilon}>0$ such that, for any sufficiently large $\mu$, we have
\begin{align*}
\|(\psi_h-\phi_h) (H_{U_h}(\mu)-z)^{-1}\phi_h u\|_{L^2( M)} \leq &
C_\epsilon e^{-(c_1-\epsilon)\mu^{1/2}} \|\phi_h
u\|_{\mu^{1/2}\Phi_h}\\ = & C_\epsilon
e^{-(c_1-\epsilon)\mu^{1/2}} \|\phi_h u\|_{L^2(U_h)}\\ = &
C_\epsilon e^{-(c_1-\epsilon)\mu^{1/2}} \| u\|_{L^2( M)}\,.
\end{align*}
By the Riesz formula, for any sufficiently large $\mu$, we have
\begin{align*}
\|(\psi_h-\phi_h) E_{H_{U_h}(\mu)} (\lambda)\phi_h u\|_{L^2( M)} \leq  
C_\epsilon e^{-(c_1-\epsilon)\mu^{1/2}} \| u\|_{L^2( M)}\,,
\end{align*}
that completes the proof of \eqref{e:E-pH}.

\section{Triviality of generalized Wannier projections}\label{s:PHU}

This section is devoted to the proof of Theorem~\ref{t:PHU}. We will keep notation introduced in Section~\ref{s:outline}.

Choose an orthonormal basis $\phi_{h,j}, j=1,2,\ldots , n_h,$ in each $H_h$. The projection $P_{\mathcal H_{\cU}}$ can be written as 
\begin{equation}\label{e:Phu}
P_{\mathcal H_{\cU}}=\sum_{h\in D}\sum_{j=1}^{n_h}\phi_{h,j}\langle \phi_{h,j}, \cdot\rangle.
\end{equation}
Given a function $f\in C_0(M)$ supported in the ball $B(x,R)$ of radius $R$ centered at some $x\in M$, we have
$$
fP_{\mathcal H_{\cU}}=\sum_{h} \sum_{j=1}^{n_h} f\phi_{h,j}\langle \phi_{h,j}, \cdot\rangle,
$$
where the sum is taken over all $h\in D$ such that $\mathcal U_{h}\cap B(x,R)\neq \emptyset$. 
As mentioned above, this set is finite. Therefore, the operator $fP_{\mathcal H_{\cU}}$ has finite rank and, therefore, is compact. Similarly, one can show that the operator $P_{\mathcal H_{\cU}}f$ is compact. Thus, the projection $P_{\mathcal H_{\cU}}$ is locally compact. For any $f,g\in C_0(M)$, we have
$$
fP_{\mathcal H_{\cU}}g=\sum_{h}\sum_{j=1}^{n_h} f\phi_{h,j}\langle \bar g\phi_{h,j}, \cdot\rangle,
$$
where the sum is taken over all $h\in D$ such that $\mathcal U_{h}\cap \operatorname{supp} f \neq \emptyset$ and $\mathcal U_{h}\cap \operatorname{supp} g\neq \emptyset$. If the distance between the supports of $f$ and $g$ is greater than $\delta$, where $\delta$ is given by \eqref{e:udiam}, then the set of such $h$'s is clearly empty, which implies that $fP_{\mathcal H_{\cU}}g=0$. It follows that $P_{\mathcal H_{\cU}}$ has finite propagation. Thus, we proved that $P_{\mathcal H_{\cU}}$ is in $C^*(M)$.

Recall some notation for Roe algebras from \cite{wannier}. Generally, the Roe algebra of a metric space $X$ is defined for an arbitrary Hilbert space representation of $C_0(X)$ satisfying certain properties, and the Roe algebra comes together with its representation on this Hilbert space. We have to be more specific. For a metric measure space $X$ and for a Hilbert space $H$ we write $C^*_H(X)$ for the Roe algebra represented on $L^2(X)\otimes H$. When $X$ is discrete, we use the notation $C^*_H(X)$ for the version of a Roe algebra determined by the canonical representation of $C_0(X)$ on $l^2(X)\otimes H$, but in this case $C^*_H(X)$ is the Roe algebra only when $H$ is infinite-dimensional. If $\dim H=1$, i.e. if $H=\mathbb C$ then $C^*_{\mathbb C}(X)$ is the uniform Roe algebra, and if $\dim H=n$ then $C^*_H(X)\cong C^*_{\mathbb C}(X)\otimes M_n(\mathbb C)$.

Now we assume that $M$ is not compact and follow the similar proof in \cite{wannier} with necessary modifications. First, by fixing a point in each $\mathcal U_{h}$, we will consider $D$ as a subset of $M$. By \eqref{e:udiscr}, this subset is uniformly discrete.
By Zorn's Lemma, we can find a maximal (with respect to inclusion) subset $D^\prime\subset M$ such that $D\subset D^\prime$ and 
\[
d\,(h_1, h_2)\geq 2\delta+\epsilon, \quad h_1\in D^\prime, \quad h_2\in D^\prime\setminus D, \quad h_1\neq h_2.
\]
Then $D^\prime $ is uniformly discrete and is a $(2\delta+\epsilon)$-net in $M$ (the latter is equivalent to the fact that $D^\prime$ is coarsely equivalent to $M$). For $h\in D^\prime\setminus D$, set $\mathcal U_{h}=B(h,r)$, the open ball in $M$ of radius $r=\min(\delta, r_M)$ centered at $h$. (Recall that $r_M>0$ is the injectivity radius of $M$.) We get a family $\mathcal U^\prime=\{\mathcal U_{h} : h \in D^\prime\}$ of relatively compact domains (with smooth boundary) in $M$, satisfying the analogs of the conditions \eqref{e:udiam} and \eqref{e:udiscr}, such that $\mathcal U\subset \mathcal U^\prime$.  

For $h\in D$, we complete the orthonormal basis $\phi_{h,j}, j=1,2,\ldots , n_h,$ in $H_h$ to an orthonormal basis $\phi_{h,j}, j=1,2,\ldots ,$ in $L^2(\mathcal U_{h})$ considered as a subspace in $L^2(M)$ and, for $h\in D^\prime\setminus D$, we choose an orthonormal basis $\phi_{h,j}, j=1,2,\ldots ,$ in $L^2(\mathcal U_{h})$.

Denote by $e_j, j=1,2,\ldots ,n,$ the standard orthonormal basis in $\mathbb C^n$. (Recall that $n$ is defined in \eqref{e:defn}.) Let $U:l^2(D^\prime)\otimes \mathbb C^n\to L^2(M)$ be the isometry defined by $$U(\delta_h\otimes e_j)=\phi_{h,j},\quad h\in D^\prime,\quad  j=1,2,\ldots , n.$$ 

The formula $T\mapsto UTU^*$ 
defines a $*$-homomorphism $j_{\mathbb C^n}:C^*(D^\prime)\otimes M_n(\mathbb C)\to \mathbb B(L^2(M))$ to the algebra of all bounded operators on $L^2(M)$. It was shown in \cite[Lemma 3]{wannier}, for $n=1$, that this $*$-homomorphism preserves the property of finite propagation, hence its range lies in $C^*_{\mathbb C}(M)\subset\mathbb B(L^2(M))$. The same argument works for any $n$, and shows that the formula $T\mapsto UTU^*$ 
defines a $*$-homomorphism $j_{\mathbb C^n}:C^*(D^\prime)\otimes M_n(\mathbb C)\to C^*_{\mathbb C}(M)$.

Note that the map $j_{\mathbb C^n}$ can be written as 
\begin{equation}\label{e:def-jCT}
j_{\mathbb C^n}(T)=\sum_{j,k=1}^n\sum_{h,h^\prime\in D^\prime}\phi_{h,j}\langle T_{hh^\prime,jk}\phi_{h^\prime,k}, \cdot\rangle,
\end{equation} 
where $T_{hh^\prime,jk}=\langle T(\delta_h\otimes e_j), \delta_{h^\prime}\otimes e_k\rangle$ are the matrix entries of $T$. 

Let $C_b(D^\prime)$ denote the commutative $C^*$-algebra of bounded functions on $D^\prime$. It is included into $C^*_{\mathbb C}(D^\prime)$ in a standard way: a function $f\in C_b(D^\prime)$ is mapped to the diagonal operator $T\in C^*_{\mathbb C}(D^\prime)$ with diagonal entries $T_{hh}=f(h)$, $h\in D^\prime$. Denote by $\gamma : C_b(D^\prime)\subset C^*_{\mathbb C}(D^\prime)$ the corresponding inclusion map. It induces a map $$\gamma_n : M_n(C_b(D^\prime))\cong C_b(D^\prime, M_n(\mathbb C)) \subset M_n( C^*_{\mathbb C}(D^\prime))\cong C^*_{\mathbb C}(D^\prime)\otimes M_n(\mathbb C).$$ 

Let $p_k\in M_n(\mathbb C)$, $k={0},1,\ldots n$, be the projection onto the first $k$ vectors of the standard basis of $\mathbb C^n$. Consider the projection $p$ in $C_b(D^\prime, M_n(\mathbb C))$ given by 
\[
p(h)=\begin{cases} p_{n_h}, & h\in D,\\ 0 & h\in D^\prime\setminus D.
\end{cases}
\]
By \eqref{e:Phu}, \eqref{e:def-jCT} and the definition of $\gamma_n$, it is easy to see that 
$$
j_{\mathbb C^n}(\gamma_n(p))=P_{\mathcal H_{\cU}}\in C^*_{\mathbb C}(M).
$$

For each $k={0},1,\ldots,n$, set $D_k=\{h\in D:n_h=k\}$. Then $D=\sqcup_{{k=0}}^n D_k$, 
\begin{equation}\label{e:p_k}
p=\sum_{k=1}^n p|_{D_k}, 
\end{equation}
and each $p|_{D_k}$ has rank $k$.

Let $V:\mathbb C\to H$ be an isometry, i.e.\ an inclusion of $\mathbb C$ onto a one-dimensional subspace of $H$, and let $V_{D^\prime}=\id\otimes V:l^2({D^\prime})=l^2({D^\prime})\otimes\mathbb C\to L^2(M)\otimes H$. We get an injective $C^*$-algebra homomorphism $i_{D'} :C^*(D^\prime)\to C^*_H(D')$ (cf. \cite[Section II]{wannier}):
\[
i_{D'}(T)=V_{D'}TV^*_{D'}, \quad T\in C^*(D^\prime).
\]
It is easy to see that 
$$
i_{D'}(T)=T\otimes e, 
$$
where $e=VV^* \in\mathbb K(H)$ is a rank one projection. This map induces a map in $K$-theory
$$
(i_{D^\prime})_*:K_0(C^*_{\mathbb C}(D^\prime))\to K_0(C^*_H(D^\prime)),
$$
which is independent of the choice of $V$. Note that the similar map $(i_M)_*:K_0(C^*_{\mathbb C}(M))\to K_0(C^*_H(M))$ is an isomorphism (and $C^*_{\mathbb C}(M)$ and $C^*_H(M)$ are isomorphic, since $M$ is not compact).  

Taking the composition of the map $\gamma_*:K_0(C_b(D^\prime))\to K_0(C^*_{\mathbb C}(D^\prime))$ induced by $\gamma$ with $(i_{D^\prime})_*$, we get a map $(\gamma_H)_*=(i_{D^\prime})_*\circ \gamma_* :K_0(C_b(D^\prime))\to K_0(C^*_H(D^\prime))$.

By \cite[Theorem 4]{wannier}, $[P_{H_{\mathcal U}}]=0$ in $K_0(C^*(M))$ iff $(\gamma_H)_*([p])=0$ in $K_0(C^*_H(D'))$, so, by (\ref{e:p_k}), it remains to show that $(\gamma_H)_*([p|_{D_k}])=0$ in $K_0(C^*_H(D'))$ for each $k=1,\ldots,n$. 

Let $D_0\subset D'$ be a subset, and let $1_{C_b(D_0)}$ denote the unit of the subalgebra $C_b(D_0)\subset C_b(D')$. 
Corollary 9 in \cite{wannier} states that $(\gamma_H)_*([1_{C_b(D_0)}])=0$ in $K_0(C^*_H(D'))$ (the argument there was based on detailed study of geometric structure of $D'$, in particular, on its ray structure). 
But ${[p|_{D_k}]}=k[1_{C_b(D_k)}]$, hence $(\gamma_H)_*([p|_{D_k}])=0$ for any $k=1,\ldots,n$, and this completes the proof of Theorem~\ref{t:PHU}. 

\section*{Acknowledgements}
The authors are grateful to the anonymous referee for valuable remarks, which allowed to improve the first version of the paper. 

The results of Sections 1--4 were obtained by Yu. Kordyukov. The work of Yu. Kordyukov is performed under the development program of Volga Region Mathematical Center (agreement No. 075-02-2024-1438). The results of Section 5 were obtained by V. Manuilov and supported by the RSF grant 24-11-00124.

\end{document}